\documentclass{scrartcl}
\usepackage[T1]{fontenc}
\usepackage{lmodern}
\usepackage{amsmath,amssymb,amsthm}

\usepackage[french]{babel}

\usepackage[utf8]{inputenc}

\usepackage{url}

\usepackage[
backend=biber,
style=authoryear, 
citestyle=authoryear-comp,
maxnames=7,
sortcites=false %%%% to keep the order in \cite command
]{biblatex}
\usepackage{csquotes}

\DeclareDelimFormat{nameyeardelim}{\addcomma\space}

\DeclareNameAlias{sortname}{given-family}

\renewcommand{\bibnamedash}{\leavevmode\raise3pt\hbox to3em{\hrulefill}\space}

\AtEveryBibitem{%
  \clearfield{issn} % Remove issn
  \clearfield{isbn} % Remove isbn
  \clearfield{doi} % Remove doi

  \ifentrytype{online}{}{% Remove url except for @online
    \clearfield{url}
  }
}

\renewbibmacro{in:}{%
    \ifentrytype{article}{}{\printtext{\bibstring{in}\intitlepunct}}}

\DeclareFieldFormat[article,periodical,inreference]{number}{\mkbibparens{#1}}
\DeclareFieldFormat[article,periodical,inreference]{volume}{\mkbibbold{#1}}
\renewbibmacro*{volume+number+eid}{%
    \printfield{volume}%
    \setunit*{\addthinspace}% NEW
    \printfield{number}%
    \setunit{\addcomma\space}%
    \printfield{eid}}

\DeclareFieldFormat[article,inbook,incollection]{title}{\enquote{#1}\addcomma} 

\DefineBibliographyExtras{french}{\restorecommand\mkbibnamefamily}
\date{Novembre 2022}
\title{Sur un théorème de Lang--Weil tordu}
\subtitle{d'après E. Hrushovski, K. V. Shuddhodan et Y. Varshavsky}

\author{Silvain Rideau-Kikuchi}
\newcommand\Aa{\mathbf{A}}
\newcommand\ACFA{\mathrm{ACFA}}
\newcommand{\restr}[2]{{\left.#1\right|_{#2}}}
\newcommand\degins{\deg_{\mathrm{ins}}}
\newcommand\Tr{\mathrm{Tr}}
\newcommand\Qq{\mathbf{Q}}
\newcommand\Zz{\mathbf{Z}}
\newcommand\Cc{\mathbf{C}}
\newcommand\Ff{\mathbf{F}}
\newcommand\id{\mathrm{id}}

\newcommand\alg[1]{\overline{#1}}
\newcommand\Pp{\mathbf{P}}
\newcommand\oX{\overline{X}}
\newcommand\bd{\partial}
\newcommand\fU{\mathfrak{U}}
\newcommand\fm{\mathfrak{m}}
\newcommand\cO{\mathcal{O}}
\newcommand\fp{\mathfrak{p}}
\newcommand\Gal{\mathrm{Gal}}
\newcommand\oC{\overline{C}}

\newcommand\of{\overline{f}}
\newcommand\tdim{\mathrm{tdim}}
\newcommand\fS{\mathfrak{S}}

\newtheorem{theo}{Theorem}[section]
\newtheorem{prop}{Proposition}
\newtheorem{coro}{Corollaire}
\newtheorem{lemm}{Lemme}

\theoremstyle{definition}
\newtheorem{defi}{Definition}
\newtheorem{exem}{Exemple}

\begin{document}

\maketitle

%\section*{Introduction}

Le fil directeur de cet exposé est la grande uniformité du comportement asymptotique, pour les grandes valeurs de \(q\), du morphisme de Frobenius
\[\begin{array}{rccc}
\phi_q: &K &\to & K\\ 
&x &\mapsto& x^q.
\end{array}\]
où \(K\) est un corps (algébriquement clos) de caractéristique positive \(p\) et \(q\) est une puissance de \(p\).

Les estimées de \textcite{LanWei} en sont un exemple classique: étant donné une variété algébrique\footnote{Dans cet exposé, on ne considèrera que des variétés \emph{irréductibles} sur un corps \(K\) algébriquement clos.} \(X\) de dimension \(d\) sur \(\alg{\Ff}_p\), pour toute puissance \(q\) de \(p\) suffisamment grande, \(X\) est définie par des équations à coefficients dans \(\Ff_q\) et le morphisme de Frobenius induit un morphisme \(\phi_{q,X} : X \to X\). Le nombre de points fixes de \(\phi_{q,X}\) est alors de l'ordre de
\[q^{d} + O(q^{d-1/2}).\]
De plus, la constante du \(O\) ne dépend que de la complexité des équations qui définissent~\(X\) --- mais pas de \(K\) ou \(p\).

Dans la mesure où les points fixes de \(\phi_{q}^n\) ne sont rien d'autre que les éléments du corps fini \(\Ff_{q^n}\), l'énoncé ci-dessus est, en fait, une estimation du nombre de points de \(X\) dans le corps \(\Ff_{q^n}\); ce qui est la manière classique de présenter le problème.\medskip

Un théorème remarquable de \textcite{Hru-Frob} donne une explication générale à ce comportement asymptotique uniforme, du moins pour ce qui est des propriétés que l'on peut exprimer par une formule du premier ordre. Ici, on considère des formules qui s'écrivent avec des symboles pour l'addition, la soustraction, la multiplication, les éléments \(0\) et \(1\) et un endomorphisme \(\sigma\) fixé et où on autorise la quantification uniquement sur les éléments des corps que l'on considère. Par exemple, la formule
\[\forall x \forall y\, \sigma(x + y) = \sigma(x) + \sigma(y) \wedge \sigma(x \cdot y) = \sigma(x)\cdot\sigma(y) \wedge\sigma(1) = 1\]
qui exprime que \(\sigma\) est un morphisme d'anneaux, ou encore les formules
\[\forall x_0 \ldots\forall x_n\, (\bigwedge_{i\leq n} \sigma(x_i) = x_i) \to \exists y\, \sigma(y) = y \wedge y^{n+1} + \sum_{i\leq n} x_i y^i = 0\]
qui expriment que le corps fixé par \(\sigma\) est algébriquement clos.

La théorie des modèles a une riche histoire de tels résultats asymptotiques.
L'un des plus anciens est un principe de transfert entre la grande
caractéristique positive et la caractéristique zéro, qui découle immédiatement
de la compacité de la logique du premier ordre. Pour toute formule \(\psi\)
\emph{sans le symbole d'automorphisme},
\begin{center}
pour tout \(p\) grand, la formule \(\psi\) est vérifiée dans \(\alg{\Ff}_p\),
\end{center}
si et seulement si,
\begin{center}
la formule \(\psi\) est vérifiée dans \(\Cc\).
\end{center}

Le théorème de Hrushovski généralise ce principe en prenant en compte le morphisme de Frobenius \(\phi_q\). D'après ce théorème, celui-ci se comporte pour les grandes valeurs de \(q\) comme un automorphisme de corps générique. Plus précisément, il existe une classe de << corps avec automorphisme générique >> dont on peut donner une axiomatisation explicite, habituellement notée \(\ACFA\) --- \emph{cf.} proposition~\ref{ACFA} pour une liste de ces axiomes. Cette classe est l'équivalent, pour les corps avec automorphisme, des corps algébriquement clos pour les corps sans automorphisme et la section~\ref{sec:ACFA} de cet exposé a pour but d'en donner une présentation plus exhaustive

Le résultat précis, que nous démontrerons dans la section~\ref{Frob}, énonce alors que pour toute formule \(\psi\),
\begin{center}
pour tout \(q\) grand\footnote{On considère bien ici toutes les puissances de tous les nombres premiers.}, la formule \(\psi\) est vraie de \(\alg{\Ff}_q\) et \(\phi_q\),
\end{center}
si et seulement si,
\begin{center}
la formule \(\psi\) est une conséquence des axiomes \(\ACFA\).
\end{center}

Ce résultat repose en grande partie sur deux résultats qui décrivent le comportement du morphisme \(\phi_q\) quand \(q\) varie. Le premier est un théorème de théorie algébrique des nombres, le théorème de densité de Chebotarev (voir, par exemple, \cite[théorème~6.3.1]{FriJar-FA}), qui décrit les liens entre le morphisme \(\phi_q\) et les extensions cycliques des corps de nombres. Le second est un résultat de géométrie algébrique qui généralise les estimées de Lang--Weil. C'est d'ailleurs un phénomène remarquable en soi que les propriétés asymptotiques au premier ordre du morphisme de Frobenius ne dépendent que de ces deux résultats, profonds, mais qui ne semblent \emph{a priori} couvrir qu'une petite partie des propriétés exprimables par des formules.

Comme on l'a vu ci-dessus, on peut voir les estimées de Lang--Weil comme un énoncé de comptage du nombre de points fixes du morphisme de Frobenius. On peut naturellement se demander si ce phénomène d'uniformité reste vrai pour des questions de comptage plus générales impliquant l'automorphisme de Frobenius. Le théorème~\ref{Twist LW} donne une réponse positive à cette question --- ainsi que son titre à l'exposé. 

Pour toute variété \(X\) sur un corps \(K\) algébriquement clos, on note \(X^{(q)} = X^{\phi_q}\) le changement de base de \(X\) le long de l'automorphisme de Frobenius et \(\phi_{q,X} : X \to X^{(q)}\) le morphisme induit par \(\phi_q\).
Si \(X\subseteq K^n\) est le lieu des zéros des polynômes \(P_1,\ldots,P_m \in K[x_1,\ldots,x_n]\), alors \(X^{(q)}\) est le lieu des zéros des polynômes \(\phi_q(P_i)\) et \(\phi_{q,X}\) est l'action de \(\phi_q\) coordonnée par coordonnée.

Les estimées de Lang--Weil tordues affirment alors que, pour toute variété \(X\) de dimension \(d\) sur \(K\), tout \(q\) suffisamment grand et toute sous-variété \(C \subseteq X\times X^{(q)}\) vérifiant certaines hypothèses techniques, l'intersection de \(C\) avec le graphe du morphisme de Frobenius \(\phi_{q,X}\) contient un nombre de points de l'ordre de
\[c q^d + O(q^{d-1/2}),\]
où \(c\) est une constante qui dépend de la géométrie de \(C\), et les différentes bornes ne dépendent que de la complexité des équations qui définissent \(X\) et \(C\).

Ce résultat est aussi originellement dû à \textcite{Hru-Frob} et sa preuve repose sur le développement de nouveaux outils, intéressants en eux-mêmes, introduits à cette occasion --- dont le développement d'une théorie des schémas aux différences ainsi que la théorie des modèles de certains corps valués avec automorphisme.
Nous exposerons, dans la section~\ref{Twist}, une preuve récente de \textcite{ShuVar} qui, pour les citer, est <<~purement géométrique~>>.\medskip

Enfin, dans la section~\ref{App}, nous discuterons de certaines applications de ces résultats en dynamique algébrique et en géométrie algébrique aux différences.
\medskip

Pour conclure cette introduction, je voudrais remercier Élisabeth Bouscaren, Antoine Chambert-Loir et Martin Hils pour nos discussions ainsi que leurs nombreux commentaires dont ce texte a grandement bénéficié.

\section{Automorphisme générique}\label{sec:ACFA}

Commençons par introduire la notion d'automorphisme << générique >> ainsi que divers outils de théorie des modèles nécessaires à sa définition.

Un objet central de ce texte est l'étude des équations dites << aux différences >>, c'est-à-dire les équations de la forme (pour le cas en une variable) \[\sum_{i_j\leq d} a_i x^{i_0} \sigma(x)^{i_1} \cdots \sigma^n(x)^{i_n} = 0,\]
où les coefficients \(a_i\) sont dans un corps \(K\) sur lequel on a choisi un endomorphisme~\(\sigma\) --- on parle alors de corps aux différences \((K,\sigma)\).
Historiquement, le nom d'équations aux différences est hérité du cas où \(K = k(t)\) est un corps de fonctions rationnelles sur un corps \(k\) et \(\sigma(f) = f(t +1)\). Un autre exemple classique est celui des équations aux \(q\)-différences où l'on considère le morphisme \(\sigma(f) = f(qt)\).
Comme on peut le voir dans ces exemples, leur étude est étroitement liée à celle de la dynamique algébrique.

Pour étudier ces équations, il est utile de disposer de << domaines universels >> dans lesquels toutes les équations aux différences ont des solutions --- voire suffisamment de solutions pour pouvoir en détecter la structure. La théorie des modèles fournit une notion abstraite d'un tel domaine:

\begin{defi}
Un corps aux différences \((K,\sigma)\) est dit \emph{existentiellement clos} si tout système d'équations aux différences sur \(K\) (en plusieurs variables) qui a une solution dans un corps aux différences \((L,\tau)\) qui contient \(K\) --- et dont l'endomorphisme~\(\tau\) étend~\(\sigma\) --- a une solution dans \(K\).
\end{defi}

En d'autres termes, pour toute formule \(\psi(x_1,\ldots,x_n)\) sans quantificateurs à paramètres dans \(K\) dans laquelle les variables non quantifiées sont parmi \(x_1,\ldots,x_n\), si \(\psi\) est vérifiée par des éléments \(a_1,\ldots,a_n\) d'une extension de \((K,\sigma)\), elle est déjà vérifiée par des élément \(c_1,\ldots,c_n\in K\).

On dit que l'automorphisme \(\sigma\) est \emph{générique} puisque tout comportement possible d'un automorphisme de corps se retrouve dans le corps aux différences \((K,\sigma)\).\medskip

Si l'on travaille seulement dans le langage des anneaux (c'est-à-dire sans le symbole pour l'endomorphisme \(\sigma\)), un corps est existentiellement clos si et seulement s'il est algébriquement clos --- c'est exactement ce qu'énonce le \emph{Nullstellensatz} de Hilbert. Il suffit donc dans ce cas de considérer des équations en une seule variable.

Les corps aux différences existentiellement clos sont, par définition, ceux qui
vérifient une forme du \emph{Nullstellensatz} pour les équations aux
différences. Mais il est aussi possible, dans ce cas, de caractériser quelles
équations aux différences doivent avoir une solution dans un corps aux
différences pour qu'il soit existentiellement clos.

Soit \(X\) une variété sur un corps aux différences \((K,\sigma)\) (algébriquement clos). On note~\(X^\sigma\) le changement de base de~\(X\) le long de~\(\sigma\) et \(\sigma_X : X \to X^\sigma\) le morphisme induit par~\(\sigma\). Comme précédemment, si \(X \subseteq  K^n\) est le lieu des zéros des polynômes \(P_1,\ldots,P_m \in K[x_1,\ldots,x_n]\), alors \(X^\sigma\) est le lieu des zéros des polynômes \(\sigma(P_i)\) obtenus en faisant agir~\(\sigma\) sur les coefficients et pour tout \((a_1,\ldots,a_n)\in X\), \(\sigma_X(a) = (\sigma(a_1),\ldots,\sigma(a_n))\).

Un morphisme \(f : X \to Y\) entre variétés est dit \emph{dominant} s'il est d'image dense pour la topologie de Zariski. On peut maintenant énoncer la caractérisation suivante, isolée par Hrushovski (\emph{cf.} \textcite[172--173]{Mac-ACFA}), des corps aux différences existentiellement clos:

\begin{prop}\label{ACFA}
Un corps aux différences \((K,\sigma)\) est existentiellement clos si et seulement si:
\begin{enumerate}
\item Le corps \(K\) est algébriquement clos;
\item Le morphisme \(\sigma\) est surjectif;
\item\label{ACFA geom} Pour toute variété affine \(X\) sur \(K\) et toute sous-variété \(C \subseteq X\times X^\sigma\) telle que les projections vers \(X\) et \(X^\sigma\) sont dominantes, l'ensemble des couples \((x,\sigma(x))\), avec \(x\in X(K)\), est Zariski dense dans \(C\).
\end{enumerate}
\end{prop}

Ces conditions peuvent s'exprimer par un ensemble (infini) de formules. La principale subtilité est de pouvoir quantifier sur les sous-variétés \(C \subseteq X\times X^\sigma\). Pour cela, il faut vérifier que, pour toute sous-variété \(X \subseteq \Aa_K^{n+m}\), l'ensemble des \(y \in \Aa_K^n\) tels que la fibre \(X_y\subseteq \Aa_K^m\) est (géométriquement) irréductible est donné par une formule --- en d'autres termes, que le lieu d'irréductibilité (géométrique) de la famille
\(X_y\) est constructible. Mais cela est bien connu, voir par exemple \textcite[théorème~9.7.7]{EGA4.3} ou \textcite{vdDSch} pour une approche
de cette question par le biais d'algèbres de polynômes non standards.
La seconde difficulté est de détecter, par une formule, quand un morphisme est dominant. Mais il suffit, pour cela, de savoir que la dimension d'une fibre est continue pour la topologie constructible --- voir, par exemple, \textcite[théorème~9.5.5]{EGA4.3}.

La classe des corps aux différences existentiellement clos admet donc une
axiomatisation (infinie), qui est habituellement notée \(\ACFA\). Cette
axiomatisation n'est pas complète, c'est-à-dire qu'il existe des corps aux
différences \((K,\sigma)\) et \((L,\tau)\) existentiellement clos qui ne
vérifient pas les mêmes formules sans variable non quantifiée --- on parle
habituellement d'énoncés. Pour qu'ils vérifient les mêmes énoncés, il faudrait
évidemment que \(K\) et \(L\) aient la même caractéristique. Mais il faut
également que les restrictions des automorphismes à la clôture algébrique de
leur corps premier soient conjuguées. \textcite[173--174]{Mac-ACFA} montre que
c'est, en fait, suffisant:

\begin{prop}\label{compl ACFA}
Soient \((K,\sigma)\) et \((L,\tau)\) des corps aux différences existentiellement clos, \(F\leq K\) un sous-corps aux différences algébriquement clos et \(f \colon F \to L\) un morphisme de corps aux différences --- c'est-à-dire un morphisme de corps tel que \(\sigma \circ f = f\circ \tau\). Alors pour toute formule \(\psi(x_1,\ldots,x_n)\) et tout \(a \in F^n\),
\[\text{\(a\) réalise \(\psi\) dans \((K,\sigma)\) si et seulement si \(f(a)\) réalise \(\psi\) dans \((L,\tau)\).}\]

En particulier, deux corps aux différences existentiellement clos \((K,\sigma)\) et \((L,\tau)\) vérifient les mêmes énoncés si et seulement si:
\[\text{les corps aux différences }(K_0,\restr{\sigma}{K_0})\text{ et } (L_0,\restr{\tau}{L_0})\text{ sont isomorphes,}\]
où \(K_0\) (respectivement \(L_0\)) est la clôture algébrique du corps premier de \(K\) (respectivement \(L\)).
\end{prop}

%Par des considérations abstraites de théorie des modèles, la proposition~\ref{compl ACFA} découle du lemme algébrique facile suivant:
%
%\begin{lemm}
%Soient \((K,\sigma)\) et \((L,\tau)\) des corps aux différences, \(F\leq K\) un sous-corps clos aux différences algébriquement clos. Tout morphisme \(f \colon F \to L\) de corps aux différences s'étend alors en un morphisme \(g \colon K \to M\) où \((M,\rho)\) est une extension de \((L,\tau)\).
%\end{lemm}
%
%\begin{proof}
%Il suffit de considérer le corps des fractions de l'anneau \(K\otimes_F L\) muni du morphisme injectif \(\sigma\otimes_F \tau\).
%\end{proof}

Le résultat précédent n'est pas exactement un résultat d'élimination des quantificateurs, mais il implique tout de même que les formules ont toutes une forme très simple, à équivalence près:

\begin{coro}
Toute formule \(\psi(x_1,\ldots,x_n)\) est équivalente, modulo les axiomes de \(\ACFA\), à une disjonction de formules de la forme:
\[\exists y\, \theta(x_1,\ldots,x_n,y),\]
où \(\theta\) est sans quantificateurs et, pour tous \(a_1,\ldots a_n\), il n'y a qu'un nombre fini uniformément borné de \(y\) qui vérifient \(\theta(a_1,\ldots,a_n,y)\).
\end{coro}

La proposition~\ref{ACFA} a aussi pour conséquence que le corps fixé par \(\sigma\) dans un corps aux différences existentiellement clos est un corps dit <<~pseudo-fini~>>:

\begin{coro}
Soient \((K,\sigma)\) un corps aux différences existentiellement clos et \(K^\sigma = \{x \in K \mid \sigma(x) = x\}\) le corps fixé par \(\sigma\). On a alors que:
\begin{itemize}
\item Le corps \(K^\sigma\) est parfait;
\item Le groupe de Galois absolu de \(K^\sigma\) est engendré par \(\sigma\) et est donc isomorphe à \(\hat\Zz\);
\item Toute variété (géométriquement intègre) sur \(K^\sigma\) a un point dans \(K^\sigma\).
\end{itemize}
\end{coro}

Cette classe doit son nom à \textcite{Ax-Psf} qui a démontré qu'un énoncé dans
le language des corps est vrai dans les corps finis de grand cardinal si et
seulement s'il est vrai dans les corps pseudo-finis. Comme les corps finis sont
exactement les corps fixés par les morphismes de Frobenius \(\phi_q\), le
résultat d'Ax est un cas particulier, ainsi qu'une des inspirations majeures, du
théorème de \textcite{Hru-Frob} dont nous discutons dans cet exposé.\medskip

La théorie des modèles des corps aux différences existentiellement clos est un sujet très riche qui va bien au-delà des quelques résultats fondamentaux, mais relativement élémentaires,
que nous utilisons ici. L'un des résultats centraux de cette théorie est le théorème de trichotomie de \textcite{ChaHru-ACFA} et \textcite{ChaHruPet} qui donne une description fine de la géométrie des ensembles de dimension \(1\) définissables par une formule. Ces résultats ont aussi des conséquences remarquables en dynamique algébrique, comme, par exemple, les travaux de \textcite{MedSca}. Mais leur exposition nous emporterait bien loin de notre sujet principal.

\section{Estimées de Lang-Weil tordues}\label{Twist}

Revenons maintenant à la question du comportement asymptotique du morphisme de Frobenius. Soient \(K\) un corps algébriquement clos de caractéristique strictement positive \(p\) et \(q\) une puissance de \(p\). D'après la proposition~\ref{ACFA}, pour montrer que \(\phi_q\) se comporte comme un automorphisme générique, il faut, \emph{a minima}, montrer que la condition~\ref{ACFA geom} est vraie, pour \(q\) grand. C'est le but de cette section, dans laquelle, nous montrons une estimation précise (voir théorème~\ref{Twist LW}) du nombre de points dans une telle intersection en suivant une preuve récente <<~purement géométrique~>> due à \textcite{ShuVar}.\medskip

Si \(f \colon X \to Y\) est un morphisme dominant entre variétés sur \(K\) de même dimension, on note \(\deg(f) = [K(X) : K(Y)]\); c'est le nombre de points (comptés avec multiplicités) dans une fibre générique de \(f\). On note aussi \(\degins(f)\) son degré inséparable; c'est le degré \([K(X): L]\) où \(L\) est la sous-extension séparable maximale de \(K(Y)\) dans \(K(X)\).
On note aussi \(\Gamma_{q,X} \subseteq X\times X^{(q)}\) le graphe de \(\phi_{q,X}\).

\begin{theo}\label{Twist LW}
Soient \(X\) une variété de dimension \(d\) sur un corps \(K\) algébriquement clos de caractéristique \(p\), \(q\) une puissance de \(p\) et \(C \subseteq X\times X^{(q)}\) une sous-variété telle que les projections \(p_1 \colon C \to X\) et \(p_2 \colon C \to X^{(q)}\) sont dominantes et \(p_2\) est quasi-finie. Alors, si \(q\) est suffisamment grand, \(\# C\cap\Gamma_{q,X} (K)\) est fini et
\[\# C\cap\Gamma_{q,X} (K) = \frac{\deg(p_1)}{\degins(p_2)} q^{d} + O(q^{d-1/2}).\]

De plus, les diverses bornes ne dépendent que de la complexité des équations qui définissent \(X\) et \(C\). Pour être précis, si \(X\subseteq \Pp_K^n\) est localement fermée, les bornes ne dépendent que de \(n\) et des degrés de la clôture \(\oX\) de \(X\), de \(\oX\setminus X\), de la clôture \(\oC\) de \(C\) et de \(\oC\setminus C\), en particulier, elles ne dépendent pas du corps \(K\) ou de sa caractéristique.
\end{theo}

\begin{exem}
Dans le cas où \(X\) est définie sur \(\Ff_q\), on identifie \(X\) à \(X^{(q)}\). Soit alors \(C = \Delta\subseteq X\times X\) la diagonale, on a
\[\Delta\cap\Gamma_{q,X}(K) = \{(x,x) : x\in X(K) \text{ et } x = \phi_q(x)\}.\]
Comme dans ce cas-là \(\deg(p_1) = \degins(p_2) = 1\), on retrouve les estimées de Lang--Weil:
\[\# X(\Ff_q) = q^d + O(q^{d-1/2}).\]
\end{exem}

\begin{exem}
Considérons l'exemple, très simple, où \(X\) est la droite affine \(\Aa_K^1\) et \(C \subseteq \Aa_K^2\) est l'ensemble \(\{(x,y) : x^m = y^{n}\}\) pour un choix d'entiers \(m\) et \(n\). Dans ce cas-là, on a \(\deg(p_1) = n\) et \(\degins(p_2) = p^r\) où \(r\) est l'exposant de \(p\) dans la décomposition de \(m\) en facteurs premiers --- et donc \(m = p^r s\) où \(s\) est un entier premier à \(p\). Alors, si \(q \geq p^r\),
\begin{align*}
\# C\cap \Gamma_{q,X}(K) & = \#\{x \in K : x^{p^r s}= x^{qn}\}\\
&= \#\{x \in K : x^{p^r(p^{-r}qn -s)} = 1\text{ ou }x=0\}\\
&= \frac{n}{p^r} q - s +1.
\end{align*}
\end{exem}

Des estimations quantitatives du théorème~\ref{Twist LW}, on déduit aisément l'énoncé qualitatif suivant, qui n'est pas sans rappeler la condition~\ref{ACFA geom} de la proposition~\ref{ACFA}:
\begin{coro}\label{Twist LW qual}
Soient \(X\) une variété sur \(K\) et \(C \subseteq X\times X^{(q)}\) une sous-variété telle que les projections \(p_1 \colon C \to X\) et \(p_2 \colon C \to X^{(q)}\) sont dominantes. Alors pour tout ouvert de Zariski \(U\subseteq C\) non-vide, si \(q\) est suffisamment grand, \[U\cap\Gamma_{q,X}(K)\neq \emptyset.\]
\end{coro}
La borne sur \(q\) ne dépend, de nouveau, que de la complexité des équations qui définissent \(X\), \(C\) et \(U\) --- en particulier, elle ne dépend ni de \(K\), ni de sa caractéristique.\medskip

Dans le cas où \(X\) est définie sur \(\Ff_q\) et \(K = \alg{\Ff}_q\) (et donc \(X^{(q)}\) est naturellement identifié à \(X\)), ces résultats sont mieux compris. Par exemple, \textcite{Shu-Frob} démontre des résultats plus fins sur le comportement de la suite des \(\# C\cap\Gamma_{q^n,X}(K)\) quand \(n\) croît :

\begin{theo}
Soient \(X\) une variété sur \(\alg{\Ff}_q\) définie sur \(\Ff_q\) et \(C \subseteq X\times X\) une sous-variété telle que \(p_2\) est quasi-finie. Pour tout \(n\) suffisamment grand, \(a_n = \# C\cap\Gamma_{q^n,X}(\alg{\Ff}_q)\) est fini et la série entière
\[\sum_{n} a_n t^n \in \Zz[[t]]\]
est rationnelle --- c'est-à-dire un élément de \(\Qq(t)\).
\end{theo}

Mais nous n'aborderons pas ici ces raffinements.

\subsection{Le cas projectif lisse}

Pour démontrer le théorème~\ref{Twist LW}, nous allons tout d'abord nous concentrer sur le cas où \(X\) est lisse projective et \(p_2\) est étale. %qui est déjà traité dans \textcite{Hru-Frob}.
Cela nous permettra d'introduire certains des principaux outils de la preuve.

\medskip

À toute variété projective lisse \(X\) de dimension \(d\) et tout \(n\leq d\), on associe le groupe \(Z^n(X)\) des cycles de \(X\) de codimension \(n\). C'est le groupe abélien libre engendré par les sous-variétés de \(X\) de codimension \(n\).

Étant donné deux sous-variétés \(Y_1\) et \(Y_2\) de \(X\) de codimensions respectives \(n\) et \(d-n\), on souhaite leur associer un entier qui représente le nombre de points dans leur intersection. Quand cette intersection est finie, il est naturel de considérer le nombre de ses points (comptés avec multiplicité).

En général, on s'autorise à d'abord << déplacer >> \(Y_1\) et \(Y_2\) pour que leur intersection soit de dimension zéro. On obtient alors un produit d'intersection \[\cdot \colon A^{n}(X) \times A^{d-n}(X) \to \Zz,\] où \(A^n(X)\) est le quotient de \(Z^n(X)\) par équivalence rationnelle. %ref ?
Pour toute sous-variété \(Y \subseteq X\) de codimension \(n\), on note \([X]\) sa classe dans  \(A^n(X)\).

Dans le cas qui nous intéresse, comme \(p_2\) est étale, les sous-variétés \(C\) et \(\Gamma_{q,X}\) de \(X\times X^{(q)}\) s'intersectent transversalement --- c'est-à-dire que l'intersection est finie et que chaque point est de multiplicité égale à un --- et donc
\begin{equation}
\# C\cap\Gamma_{q,X} (K) = [C]\cdot[\Gamma_{q,X}].
\end{equation}

L'intérêt de se ramener au calcul d'un nombre d'intersection est que l'on dispose alors d'outils cohomologiques.

Soit \(\ell\) un nombre premier différent de \(p\). À toute variété (propre et lisse) \(X\), on associe ses groupes de cohomologie \(\ell\)-adique \(H^i(X,\Qq_\ell)\) --- ce sont des \(\Qq_\ell\)-espaces vectoriels de dimension finie \(\beta_i(X)\). De plus, pour tout morphisme de variétés (propres et lisses) \(f \colon X \to Y\), on a un morphisme de tiré en arrière \(f^\star \colon H^i(Y,\Qq_\ell)\to H^i(X,\Qq_\ell)\). Si \(\dim(Y) = \dim(X)\), on a aussi un morphisme de poussé en avant  \(f_\star \colon H^i(X,\Qq_\ell)\to H^{i}(Y,\Qq_\ell)\).
On peut étendre cette fonctorialité aux sous-variétés irréductibles \(C\subseteq X\times Y\) de même dimension que \(Y\) --- et donc de codimension \(\dim(X)\) --- en posant
\[H^i([C]) = (p_2)_\star \circ p_1^\star \colon H^i(X,\Qq_\ell) \to H^i(C,\Qq_\ell) \to H^i(Y,\Qq_\ell),\]
où \(p_1 \colon C \to X\) et \(p_2 \colon C \to Y\) sont les projections. On étend \(H^i\) à tous les cycles de codimension \(\dim(X)\) par linéarité.

On peut alors écrire la formule de Grothendieck-Lefschetz, voir par exemple \textcite[proposition~3.3]{SGA4h}, qui relie le nombre d'intersection à la trace de divers morphismes de cohomologie. Pour tout \(\alpha \in A^d(X\times X^{(q)})\), on a

\begin{equation}\label{GroLef}
\alpha\cdot[\Gamma_{q,X}] = \sum_{i=0}^{2d} (-1)^i \Tr(\phi_{q,X}^{\star} \circ H^i(\alpha)).
\end{equation}

Il nous faut donc estimer chacun des termes de la somme dans le cas \(\alpha = [C]\). Pour ce qui est du degré \(i = 2d\), \(H^{2d}(X)\) est de dimension \(1\) et on peut calculer explicitement, par exemple en considérant l'image d'un cycle de dimension zéro, que 
\[\phi_{q,X}^\star \circ H^{2d}([C]) = \phi_{q,X}^\star \circ (p_2)_\star \circ p_1^\star = q^d \deg(p_1) \id\]
et donc:
\begin{equation}\label{top dim}
\Tr(\phi_{q,X}^\star \circ H^{2d}([C])) = \deg(p_1) q^{d}.
\end{equation}

Il reste donc à montrer que les contributions des termes de degré \(i < 2d\) sont négligeables. On fixe dorénavant un plongement \(\iota \colon \alg{\Qq}_\ell \to \Cc\) et on identifie \(\alg{\Qq}_\ell\) à un sous-corps de \(\Cc\). Pour tout \(f \colon H^i(X,\Qq_\ell) \to H^i(X,\Qq_\ell)\), on note \(\rho(f)\) --- ou \(\rho(f,H^i(X,\Qq_\ell))\) s'il peut y avoir une ambiguïté --- son rayon spectral (archimédien). C'est le maximum des normes des valeurs propres de \(f\). On a alors
\[|\Tr(\phi_{q,X}^\star \circ H^{i}([C]))| \leq \beta_i(X) \rho(\phi_{q,X}^\star \circ H^{i}([C]))  = O(\rho(\phi_{q,X}^\star \circ H^{i}([C]))),\]
où \(\beta_i(X)\) est la dimension de \(H^i(X,\Qq_\ell)\).

Si \(X\) et \(C\) sont définis sur \(\Ff_q\), on peut naturellement identifier \(X^{(q)}\) à \(X\) et \(\phi_{q,X}^\star\) commute alors avec \(H^i([C])\). On a donc \(\rho(\phi_{q,X}^\star \circ H^{i}([C])) \leq \rho(\phi_{q,X}^\star,H^i(X,\Qq_\ell))\rho(H^{i}[C])\) et le théorème~\ref{Twist LW} découle alors du théorème de pureté de \textcite{Del-Weil}:

\begin{theo}
Pour toute variété projective lisse \(X\) de dimension \(d\) définie sur \(\Ff_q\) et tout \(i\leq 2d\), les valeurs propres de \(\phi_{q,X}^\star\) sur \(H^i(X,\Qq_\ell)\) sont des nombres algébriques de norme complexe \(q^{i/2}\). En particulier,
\[\rho(\phi_{q,X}^\star,H^i(X,\Qq_\ell)) = q^{i/2}.\]
\end{theo}

Dans le cas où \(X\) n'est pas définie sur \(\Ff_q\), \textcite{Hru-Frob} introduit un dernier ingrédient: une norme sur les cycles.

\begin{defi}
Soit \(\alpha\) un élément de \(A^\star(X\times X^{(q)})\). On définit
\[|\alpha| = \min_{\sum_i a_i [Y_i] = \alpha} \sum_i |a_i| \deg(Y_i)\]
et
\[|H^i(\alpha)| = \min_{H^i(\gamma) = H^i(\alpha)} |\gamma|.\]
\end{defi}

L'intérêt de cette norme est qu'elle est sous-multiplicative pour la composition (à renormalisation près) et qu'elle permet de borner le rayon spectral.

\begin{prop}[{\textcite[lemme~10.17.3]{Hru-Frob}}]
Pour toutes variétés projectives lisses \(X_1\), \(X_2\) et \(X_3\) de dimension \(d\) et tout \(\alpha_1 \in A^d(X_1\times X_2)\) et \(\alpha_2 \in A^d(X_2\times X_2)\),
\[|H^i(\alpha_2) \circ H^i(\alpha_1)| = O(|H^i(\alpha_1)|\cdot |H^i(\alpha_2)|),\]
où les bornes ne dépendent que du degré des variétés \(X_1\), \(X_2\) et \(X_3\).
\end{prop}

%éléments de preuve?

Quitte à ne considérer dorénavant que des variétés de degré borné, on définit
\[||H^i(\alpha)|| = N |H^i(\alpha)|,\]
avec \(N\) suffisamment grand pour que cette norme soit sous-multiplicative. \textcite[proposition~11.11]{Hru-Frob} démontre alors que, pour tout \(\alpha \in A^d(X\times X)\),
\[\rho(H^i(\alpha))\leq ||H^i(\alpha)||.\]

%éléments de preuve?

On trouve aussi une présentation alternative de ces résultats dans \textcite[appendice~B]{ShuVar}.\medskip

Nous pouvons maintenant conclure la preuve du théorème~\ref{Twist LW} dans le
cas où \(K = \alg{\Ff}_q\). Dans ce cas, il existe un entier \(r\) tel que \(X\)
et \(C\) soient définis sur \(\Ff_{q^r}\). On peut alors identifier
\(X^{(q^r)}\) à \(X\) et \(C^{(q^r)}\) à \(C\) et on vérifie que
\[\left(\phi_{q,X}^\star\circ H^i([C])\right)^r = \phi_{q^r,X}^\star \circ h = h \circ  \phi_{q^r,X}^\star,\]
où \(h= H^i([C^{(q^{r-1})}]) \circ\cdots\circ H^i([C]) = H^i([C^{(q^{r-1})}] \circ\cdots\circ [C])\), et donc
\[\rho(\phi_{q,X}^\star\circ H^i([C]))^r = \rho(\phi_{q^r,X}^\star \circ h) = \rho(h) q^{ir/2} \leq ||H^i([C])||^{r} q^{ir/2}.\]
ce qui conclut la preuve quand \(K = \alg{\Ff}_q\). Le cas général s'en déduit par un argument de spécialisation.\medskip

Par linéarité, nous avons, en fait, prouvé une estimation de la trace pour tous les cycles de dimension \(d\), qui sera utile par la suite. Nous rappelons que nous avons identifié \(\alg{\Qq}_\ell\) à un sous-corps de \(\Cc\).

\begin{prop}
Pour toute variété projective lisse \(X\) sur \(K\), tout \(\alpha\in A^d(X\times X^{(q)})\) et tout \(i\leq 2d\),
\begin{equation}\label{approx trace}
\Tr(\phi_{q,X}^\star\circ H^i(\alpha)) = O(q^{i/2}).
\end{equation}
\end{prop}

\subsection{Un cas intermédiaire}

Pour introduire un dernier ingrédient important de la preuve, nous allons considérer le cas où \(X \subseteq \Pp_K^n\) est localement fermée et lisse. On suppose de plus que:
\begin{itemize}
\item la clôture \(\oX\) de \(X\) dans \(\Pp_K^n\) est lisse;
\item le bord \(\bd X = \oX\setminus X\) est une union finie de diviseurs lisses \((X_i)_{i\in I}\) à croisements normaux.
\end{itemize}

Pour tout \(J\subseteq I\), on note \(X_J = \bigcap_{i\in J} X_i\). %X_J est projective lisse de dim d - |J|?
En suivant une construction de \textcite{Pin}, on considère \(\pi \colon \tilde Y \to \oX\times\oX^{(q)}\) l'éclatement de \(\oX\times\oX^{(q)}\) le long de \(\bigcup_i X_i \times X_i^{(q)}\). Soit \(\tilde\Gamma\subseteq \tilde Y\) la clôture de \(\pi^{-1}(\Gamma_{q,X})\). Pour tout \(\tilde\alpha\in A^d(Y)\),  \textcite[lemme~2.3.1]{ShuVar} démontrent l'égalité suivante inspirée par des travaux de \textcite[proposition~IV.6]{Laf}:
\begin{equation}\label{incl excl}
\tilde\alpha\cdot [\tilde\Gamma] = \sum_{J\subseteq I} (-1)^{|J|} \alpha_J \cdot [\Gamma_{q,X_J}],
\end{equation}
où les \(\alpha_J \in A^{d-|J|}(X_J\times X_J^{(q)})\) sont construits explicitement, même si cela ne sera pas utile dans la suite de cet exposé (sauf dans le cas de \(\alpha_\emptyset\)).  Pour tous \(J\subseteq I\), soient \(E_J = \pi^{-1}(X_J\times X_J^{(q)})\), \(i_J \colon E_J \to Y\) l'inclusion et \(\pi_J = \restr{\pi}{E_J}\) la restriction. On a alors \(\alpha_J = (\pi_J)_\star i_J^\star(\tilde \alpha)\). En particulier, \(\alpha_\emptyset = \pi_\star(\tilde\alpha)\).

Soit \(\tilde C \subseteq \tilde Y\)  la clôture de \(\pi^{-1}(C)\) dans \(\tilde Y\). Pour conclure la preuve, il nous reste à relier \(\# C\cap\Gamma_{q,X}(K)\) à \([\tilde C]\cdot [\tilde\Gamma]\). Pour cela, il nous faut faire une dernière hypothèse sur~\(C\).

\begin{defi} Soient \(X\) une variété sur \(K\) et \(C\subseteq X\times X^{(q)}\) et \(Y\subseteq X\) des sous-variétés. On note \(p_1 \colon C \to X\) et \(p_2 \colon C \to X^{(q)}\) les projections.
\begin{itemize}
\item On dit que \(Y\) est \(C\)-invariante si \(p_1(p_2^{-1}(Y(K))) \subseteq Y(K)\).
\item On dit que \(Y\) est localement \(C\)-invariante si, pour tout \(x\in Y(K)\), il existe un voisinage ouvert (de Zariski) \(U \subseteq X\) de \(x\) tel que \(Y\cap U\) soit \(\restr{C}{U}\)-invariant, où \(\restr{C}{U} = C\cap (U\times U^{(q)})\).
\end{itemize}
\end{defi}

\begin{exem}
Si \(C\) est le graphe d'un morphisme \(f \colon X \to X^{(q)}\), une sous-variété \(Y\subseteq X\) est \(C\)-invariante si et seulement si \(f(Y(K))\subseteq Y(K)\).
\end{exem}

Soit \(\oC\) la clôture de \(C\) dans \(\oX\times \oX^{(q)}\). En supposant \(\partial X = \oX\setminus X\) localement \(\oC\)-invariant et \(p_2\) lisse, \textcite[lemme~2.3.2]{ShuVar} montrent que
\begin{align*}
\# C\cap\Gamma_{q,X}(K) &= [\tilde C]\cdot [\Gamma]\\
&= \sum_{J\subseteq I} (-1)^{|J|} [\tilde C]_J \cdot [\Gamma_{q,X_J}] &\text{par (\ref{incl excl})}\\
&= \sum_{J\subseteq I} (-1)^{|J|} \sum_{i=0}^{2(d-|J|)} (-1)^i \Tr(\phi_{q,X_J}^{\star} \circ H^i([\tilde C]_J)) &\text{par (\ref{GroLef})}\\
&= \Tr(\phi_{q,\oX}^\star\circ H^{2d}( [\tilde C]_\emptyset)) + \sum_{i < 2d} O(q^{i/2})&\text{par (\ref{approx trace}})\\
&= \deg(p_1) q^{d} + O(q^{d-1/2}),
\end{align*}
où la dernière égalité suit du fait que \([\tilde C]_{\emptyset} = [\oC]\) et de (\ref{top dim}).

\subsection{Le cas général}

Le cas général s'obtient (presque) par réduction au cas précédent. Comme
précédemment, on note \(\oX\) la clôture de Zariski de \(X\subseteq\Pp_K^n\) et
on note \(\oC\) la clôture de \(C\) dans \(\oX\times\oX^{(q)}\). Quitte à se
retreindre à des ouverts denses de \(X\), on se ramène au cas où,
\emph{cf.} \textcite[proposition~3.2.1]{ShuVar} :
\begin{itemize}
\item la variété \(X\) est lisse;
\item le bord \(\partial X = \oX\setminus X\) est localement \(\oC\)-invariant;
\item le morphisme \(p_2 \colon C \to X^{(q)}\) s'écrit comme la composée d'un homéomorphisme universel plat et d'un morphisme étale.
\end{itemize}

La principale difficuté est d'assurer la seconde condition, ce que le lemme suivant permet:

\begin{lemm}[{\cite[proposition~1.1.7]{ShuVar}}]
Il existe un ouvert dense
\(U\subseteq X\) et un éclatement \(\pi : \tilde X \to \oX\) qui est un
isomorphisme sur \(U\) tel que  \(\tilde X\setminus \pi^{-1}(U)\) est localement
\(\tilde C\)-invariant, où \(\tilde C \subseteq \tilde X\times \tilde
X^{(q)}\) est la clôture de \(\pi^{-1}(\restr{\oC}{U})\).
\end{lemm}

D'après un théorème de \textcite[théorème~4.1]{dJo}, soit alors \(Z\) une
variété projective lisse de même dimension que \(X\) et \(f \colon Z \to \oX\)
une altération  --- c'est-à-dire un morphisme dominant propre génériquement fini
--- telle que \(f^{-1}(\partial X) \subseteq Z\) soit un diviseur à croisements
normaux strict. Soient \(\pi \colon \tilde Y \to Z\times
Z^{(q)}\) l'éclatement de \(Z\times Z^{(q)}\) le long de \(f^{-1}(\partial X)\)
et \(\tilde \Gamma\subseteq \tilde Y\) la clôture de \(\pi^{-1}(\Gamma_{q,Z})\).
\textcite[\S 2.2.5]{ShuVar} construisent alors \(\tilde \alpha \in A^d(\tilde
Y)\) tel que:
\begin{align*}
\deg(f)\degins(p_2) \cdot \# C\cap\Gamma_{q,X}(K) &= \tilde\alpha\cdot [\tilde\Gamma]\\
&= \deg(f)\deg(p_1)q^d + O(q^{d-1/2}).
\end{align*}

\section{Théorie asymptotique de l'automorphisme de Frobenius}\label{Frob}

Nous pouvons maintenant prouver le théorème de \textcite{Hru-Frob} selon lequel la théorie asymptotique du Frobenius est celle d'un automorphisme générique. Commençons par en rappeler l'énoncé:

\begin{theo}\label{asymp Frob}
Pour tout énoncé \(\psi\),
\begin{center}
pour tout \(q\) grand, \(\psi\) est vraie de \(\alg{\Ff}_q\) et \(\phi_q\) --- ce qu'on note \((\alg{\Ff}_q,\phi_q) \models \psi\),
\end{center} 
si et seulement si,
\begin{center}
\(\psi\) est conséquence des axiomes \(\ACFA\) --- ce qu'on note \(\ACFA\models\psi\).
\end{center}
\end{theo}

Commençons par reformuler ce résultat en terme d'ultraproduits, qui sont des objets naturels pour traiter ces questions asymptotiques.

\begin{defi}
Soit \(I\) un ensemble non vide. Un ultrafiltre non-principal sur \(I\) est un ensemble non vide \(\fU\) de parties de \(I\), tel que:
\begin{itemize}
\item L'ensemble \(\fU\) ne contient aucun ensemble fini;
\item L'ensemble \(\fU\) est clos par intersection finie;
\item Pour tout \(X\subseteq I\), si \(X \not\in \fU\) alors \(I\setminus X \in \fU\).
\end{itemize}
\end{defi}

Par le lemme de Zorn, toute collection de parties de \(I\) dont les intersections finies sont infinies est incluse dans un ultrafiltre non-principal.

\begin{defi}
Soient \(\fU\) un ultrafiltre non-principal sur \(I\) et \((K_i,\sigma_i)\) des corps aux différences indexés par \(I\). L'ultraproduit \(\prod_{i\to\fU}(K_i,\sigma_i)\) est le corps aux différences
\[(\prod_i K_i / \fm,\sigma),\]
où \(\fm\) est l'idéal (maximal) de \(\prod_i K_i\) tel que, pour tout \(x_i\in K_i\),
\[(x_i)_{i} \in \fm\text{ si et seulement si } \{i \in I : x_i = 0\}\in\fU\]
et \(\sigma\) est l'automorphisme induit par les \(\sigma_i\):
\[\sigma((x_i)_{i} + \fm) = (\sigma_i(x_i))_i + \fm.\]
\end{defi}

D'après un théorème de \L o\v s,
 pour tout énoncé \(\psi\), on a
\begin{equation}\label{Los}
\prod_{i\to\fU} (K_i,\sigma_i) \models\psi\text{ si et seulement si }\{i\in I : (K_i,\sigma_i) \models \psi\} \in \fU.
\end{equation}

On voit alors que si \(\{q : (\alg{\Ff}_q,\phi_q) \models \psi\}\) est cofini, il est contenu dans tout ultrafiltre non-principal \(\fU\) sur l'ensemble des puissances de nombres premiers et on a donc \(\prod_{q\to\fU} (\alg{\Ff}_q,\phi_q) \models \psi\). Réciproquement, si \(X = \{q : (\alg{\Ff}_q,\phi_q) \not\models \psi\}\) est infini, il existe un ultrafiltre non-principal \(\fU\) tel que \(X\in\fU\). On a alors \(\prod_{q\to\fU} (\alg{\Ff}_q,\phi_q) \not\models \psi\). Le théorème~\ref{asymp Frob} est donc équivalent à l'énoncé suivant:

\begin{prop}\label{ultra Frob}
Pour tout corps aux différences \((K,\sigma)\), sont équivalents:
\begin{enumerate}
\item Le corps aux différences \((K,\sigma)\) vérifie les axiomes de \(\ACFA\);
\item Il existe un ultrafiltre non-principal \(\fU\) sur l'ensemble des puissances de nombres premiers tel que \(\prod_{q\to\fU} (\alg{\Ff}_q,\phi_q)\) et \((K,\sigma)\) vérifient les mêmes énoncés.
\end{enumerate}
\end{prop}

Pour prouver que l'assertion 2 implique l'assertion 1, il faut vérifier que pour tout ultrafiltre non-principal \(\fU\) sur les puissances de nombres premiers, l'ultraproduit \((K,\sigma) = \prod_{q\to\fU} (\alg{\Ff}_q,\phi_q)\) vérifie les axiomes de \(\ACFA\) (voir la proposition~\ref{ACFA}). Tout d'abord, le corps \(K\) est algébriquement clos, et \(\sigma\) est surjectif, puisque cela s'exprime par des énoncés et que c'est le cas pour tous les corps aux différences \((\alg{\Ff}_q,\phi_q)\).

Soit enfin \(\psi_{n}\) l'énoncé qui exprime que, pour toutes variétés \(X \subseteq \Aa^n\) et \(C \subseteq X\times X^\sigma\) définies par au plus \(n\)~équations de degré au plus~\(n\) et telles que les projections de~\(C\) vers~\(X\) et~\(X^\sigma\) sont dominantes, et tout ouvert de Zariski \(U \subseteq C\) non vide défini par au plus \(n\)~équations de degré au plus~\(n\), il existe un \(x\in X\) tel que \((x,\sigma(x)) \in U\). Par le corollaire~\ref{Twist LW qual}, l'ensemble des~\(q\) tel que \(\{(\alg{\Ff}_q,\phi_q) \models \psi_n\}\) est cofini. Il est donc dans tous les ultrafiltres non-principaux~\(\fU\) et donc \(\prod_{q\to\fU} (\alg{\Ff}_q,\phi_q) \models \psi_n\).

La réciproque est plus classique et déjà essentiellement présente dans les travaux de \textcite{Ax-Psf}. Étant donné un corps aux différences \((K,\sigma)\) existentiellement clos, d'après la proposition~\ref{compl ACFA}, il nous faut trouver un ultrafiltre non-principal \(\fU\) sur les puissances de nombres premiers tel que les clôtures algébriques des corps premiers de \(K\) et \(\prod_{q\to\fU} (\alg{\Ff}_q,\phi_q)\) soient isomorphes comme corps aux différences.

Si \(K\) est de caractéristique strictement positive \(p\), c'est relativement immédiat. Une fois choisie une identification de la clôture algébrique du corps premier de \(K\) à \(\alg{\Ff}_p\), pour tout entier \(n\geq 1\), l'ensemble
\[\Sigma_n = \{m \geq 1 : \restr{\sigma}{\Ff_{p^n}} = \phi_{p^m}\}\]
est infini. En effet, il existe un entier \(r\) tel que \(\restr{\sigma}{\Ff_{p^n}} = \phi_{p^r} = \phi_{p^{r + ms}}\), pour tout entier \(s\geq 1\). De plus, si \(n\) divise \(m\), on a \(\Sigma_m \subseteq \Sigma_n\). Il existe donc un ultrafiltre non-principal \(\fU\) sur l'ensemble des entiers strictement positifs qui contient tous les \(\Sigma_n\). Par le théorème de \L o\v s (\emph{cf.} (\ref{Los})), l'application diagonale \(\theta \colon (\alg{\Ff}_p,\restr{\sigma}{\alg{\Ff}_p}) \to \prod_{q\to\fU} (\alg{\Ff}_q,\phi_q)\) qui à tout \(x \in \alg{\Ff}_p\) associe la classe de \((x)_{n\geq 1}\) est un morphisme de corps aux différences, ce qui conclut la preuve de la proposition~\ref{ultra Frob} dans le cas de caractéristique positive.

Si \(K\) est de caractéristique nulle, la preuve est plus complexe et repose sur un résultat classique de théorie algébrique des nombres: le théorème de densité de Chebotarev.
Soit \(F \leq K\) une extension galoisienne finie de \(\Qq\) contenue dans \(K\). On note \(\cO \subseteq F\) la clôture intégrale de \(\Zz\) dans \(F\). Pour tout nombre premier \(p\) et tout idéal premier \(\fp\) de \(\cO\) au dessus de \(p\) --- c'est-à-dire tel que \(\fp \cap \Zz = (p)\) --- on définit \(D_\fp = \{\sigma\in\Gal(F:\Qq) \mid \sigma(\fp) = \fp\}\), son groupe de décomposition. Le corps \(k_\fp = \cO/\fp\) est alors une extension (galoisienne) finie de \(\Ff_p\) et on a un morphisme naturel \(D_{\fp} \to \Gal(k_\fp/\Ff_p)\). On dit que \(p\) est non ramifié (dans \(F\)), si c'est un isomorphisme. On note alors \(\phi_\fp\in D_\fp\) la préimage de \(\phi_p\). Une conséquence du théorème de Chebotarev est que l'ensemble:
\[\Sigma_{F} = \{p \text{ non ramifié} : \restr{\sigma}{F} = \phi_\fp\text{ pour un }\fp\text{ au dessus de }p\}\]
est infini --- pour être précis, le théorème de densité de Chebotarev énonce que cet ensemble est de densité (naturelle ou analytique) égale à la taille de la classe de conjugaison de \(\restr{\sigma}{F}\) dans \(\Gal(F/\Qq)\) divisée par le degré de l'extension.

Soit \(\fU\) un ultrafiltre non-principal sur l'ensemble des nombres premiers qui contient~\(\Sigma_F\). Pour tout \(p\in \Sigma_F\), on fixe un \(\fp\) au dessus de \(p\) tel que \(\restr{\sigma}{F} = \phi_\fp\) --- et si \(p\) n'est pas dans \(\Sigma_F\), on fixe \(\fp\) quelconque au dessus de \(p\). L'application \((\cO,\restr{\sigma}{\cO}) \to \prod_{p\to\fU} (\alg{\Ff}_p,\phi_p) \) qui à tout \(x\in\cO\) associe la classe de \((x + \fp)_{p}\), est un morphisme injectif d'anneaux aux différences. En effet, \(\Sigma_F \in \fU\) et pour tout \(p \in \Sigma_F\), \[\sigma(x) + \fp = \phi_\fp(x) + \fp = \phi_p(x + \fp).\]
Ce morphisme induit un morphisme de corps aux différences \((F,\restr{\sigma}{F}) \to \prod_{p\to\fU} (\alg{\Ff}_p,\phi_p)\).

Si \(E\) est une extension galoisienne de \(F\) contenue dans~\(K\), on a \(\Sigma_E \subseteq \Sigma_F\). Il existe donc un ultrafiltre non-principal~\(\fU\) qui contient tous les~\(\Sigma_F\).
On a donc montré ci-dessus que pour ce choix de \(\fU\), pour toute extension galoisienne finie~\(F\) de~\(\Qq\) contenue dans~\(K\), il existe un morphisme de corps aux différences \(\theta_F \colon (F,\restr{\sigma}{F}) \to \prod_{p\to\fU} (\alg{\Ff}_p,\phi_p)\).

Soit \(K_0\) la clôture algébrique de \(\Qq\) dans \(K\). Par compacité du groupe de Galois absolu de \(\Qq\) pour la topologie profinie (ou par la compacité de la logique du premier ordre), on peut recoller ces morphismes pour produire un morphisme de corps aux différences \(\theta \colon (K_0,\restr{\sigma}{K_0}) \to \prod_{p\to\fU} (\alg{\Ff}_p,\phi_p)\). Par la proposition~\ref{compl ACFA}, les corps aux différences \((K,\sigma)\) et \(\prod_{p\to\fU} (\alg{\Ff}_p,\phi_p)\) vérifient les mêmes énoncés, ce qui conclut la preuve de la proposition~\ref{ultra Frob} dans le cas de caractéristique zéro.\medskip

On remarque que dans la preuve ci-dessus, si \(K\) est de caractéristique strictement positive \(p\), on n'a besoin que d'un ultraproduit des \((\alg{\Ff}_p,\phi_{p^n})\); ce qu'on pourrait de toute manière déduire du théorème (\ref{Los}) de \L o\v s. Si \(K\) est de caractéristique zéro, un ultraproduit des \((\alg{\Ff}_p,\phi_{p})\) suffit, ce qui est plus inattendu.
Si, pour tout \(p\) premier, \(\ACFA_p\) est l'ensemble d'énoncés \(\ACFA \cup \{p = 0\}\) et \(\ACFA_0\) l'ensemble \(\ACFA\cup\{p \neq 0 : p\text{ premier}\}\), nous avons donc prouvé le résultat suivant qui précise le théorème~\ref{asymp Frob}:

\begin{prop}
Pour tout énoncé \(\psi\),
\[(\alg{\Ff}_p,\phi_p) \models \psi \text{ pour tout } p \text{ premier grand si et seulement si }\ACFA_0\models\psi\]
et 
\[(\alg{\Ff}_p,\phi_{p^n}) \models \psi \text{ pour tout } n \text{ grand si et seulement si }\ACFA_p\models\psi.\]
\end{prop}

\section{Quelques applications}\label{App}

Pour finir, exposons d'autres conséquences du théorème~\ref{Twist LW} en dynamique algébrique et en géométrie algébrique aux différences. Sans prétendre être exhaustif, cet énoncé a aussi des conséquences en géométrie algébrique \parencite{EsnMeh,EsnSriBos} et en théorie des groupes \parencite{BorSap}, mais elles s'éloignent plus du sujet principal de cet exposé et nous ne les aborderons pas ici.

\subsection{Dynamique algébrique}

Les premiers résultats dont nous discuterons concernent les points périodiques des endomorphismes de variétés. Soient \(X\) une variété sur un corps \(K\) algébriquement clos et \(f \colon X \dashrightarrow X\) une application rationnelle. La \(f\)-orbite d'un point \(x\)  est l'ensemble des \(f^i(x)\), où \(i\geq 0\) est un entier, quand ils sont définis. Un point \(x\) de \(X(K)\) est dit périodique s'il existe un entier \(i\geq 0\) tel que \(f^i(x)\) est défini et égal à \(x\). 

Il découle du théorème~\ref{Twist LW} que, sur \(\alg{\Ff}_p\), l'ensemble des points périodiques d'une application rationnelle dominante est dense:

\begin{prop}\label{dens period}
Soient \(X\) une variété sur \(\alg{\Ff}_p\) et \(f\colon X \dashrightarrow X\) une application rationnelle dominante. L'ensemble des points périodiques de \(f\) dans \(X(\alg{\Ff}_p)\) est alors Zariski dense dans \(X\).
\end{prop}

\begin{proof}
Soit \(q\) une puissance de \(p\) telle que \(X\) et \(f\) soient définies sur \(\Ff_q\). Par le corollaire~\ref{Twist LW qual}, l'ensemble \(\bigcup_{n\geq 1} \Gamma_f \cap \Gamma_{q^n,X}(\alg{\Ff}_q)\) est Zariski dense dans \(\Gamma_f\). En particulier, sa (première) projection sur \(X\) est Zariski dense dans \(X\). Pour tout élément \(x\) dans cette projection, \(f(x)\) est défini et il existe un entier \(n\geq 1\) tel que \(f(x) = \phi_{q,X}^n(x)\). Si \(m\) est tel que  \(x\in \Ff_{q^m}\), on a
\[f^m(x) = \phi_{q^{nm},X}(x) = x\]
et donc \(x\) est un point périodique de \(f\).
\end{proof}

Par un argument de spécialisation, \textcite[théorème~5.1]
{Fak} en déduit un résultat similaire sur un corps quelconque:

\begin{theo}
Soient \(X\) une variété projective sur un corps algébriquement clos~\(K\), \(f \colon X \to X\) un morphisme dominant. On suppose qu'il existe  un fibré en droites~\(L\) sur~\(X\) tel que \(f^\star L\otimes L^{-1}\) soit ample. Alors, l'ensemble des points périodiques de~\(f\) dans \(X(K)\) est Zariski dense dans~\(X\).
\end{theo}

Bien que cela puisse paraître au premier abord un peu contre-intuitif, \textcite{Ame} déduit aussi de la proposition~\ref{dens period} un résultat de densité des points de \(f\)-orbite infinie:

\begin{theo}
Soient \(X\) une variété sur \(\alg{\Qq}\) et \(f \colon X \dashrightarrow X\) une application rationnelle dominante qui n'est pas d'ordre fini. L'ensemble des points de \(f\)-orbite infinie dans \(X(\alg{\Qq})\) est alors Zariski dense dans \(X\).
\end{theo}

L'idée de la preuve est la suivante. Soit \(K\) un corps de nombres sur lequel \(X\) et \(f\) sont définies. On applique alors la proposition~\ref{dens period} à la réduction \(\of\) de \(f\) modulo un idéal premier bien choisi \(\fp\) de l'anneau des entiers \(\cO\) de \(K\). Quitte à remplacer \(K\) par une extension finie, on trouve donc un point périodique \(x\in \cO/\fp\) de \(\of\) --- et on peut même supposer que \(\of\) est étale en \(x\).

On considère alors l'ensemble \(U\) des éléments de \(X(K_\fp)\) qui se réduisent à \(x\) module \(\fp\), où \(K_\fp\) est le complété de \(K\) pour la topologie \(\fp\)-adique. Par construction, il est invariant par une puissance de \(f\). Par une combinaison de résultats de \textcite{AmeBogRov} et \textcite{BelGhiTuc}, on démontre alors qu'il existe une borne uniforme sur la taille des \(f\)-orbites finies de points de \(U\) --- et donc qu'elles sont contenues dans une sous-variété analytique propre. Le théorème découle alors du fait que les points de \(X(\alg{\Qq})\) sont denses dans \(U\).

%n'utilise que la version sur \(\Ff_q\)

\subsection{Géométrie algébrique aux différences}

Soit \((K,\sigma)\) un corps aux différences. Un polynôme aux différences \(P(x_1,\ldots x_n)\) à coefficients dans \(K\) est une fonction de la forme \(Q((\sigma^j(x_i))_{i\leq n,j\geq 0})\) où \(Q\in K[X_{i,j} : i\leq n,j\geq 0]\). L'ordre de \(x_i\) dans \(P\) est le \(m\) maximal tel que \(X_{i,m}\) apparaît dans un monôme de \(Q\) avec un coefficient non nul --- de manière équivalente, le \(m\) maximal tel que \(\sigma^m(x_i)\) apparaît dans \(P\). On choisit la convention que, si aucun des \(\sigma^j(x_i)\) n'apparaît dans \(P\), l'ordre de \(x_i\) est \(-\infty\).

Étant donné \(n\) polynômes aux différences en \(n\) variables \(P_j(x_1,\ldots,x_n)\), où \(1\leq j \leq n\), à coefficients dans \(K\), on souhaite étudier la géométrie de la <<~variété aux différences affine~>> \(X\) définie comme le lieu d'annulation des \(P_j\).
On définit la dimension totale \(\tdim(X)\) de \(X\) comme étant le supremum du degré de transcendance de \(L(\sigma^k(a_i): i\leq n,k\geq 0)\) sur \(L\), pour toute extension \((K,\sigma)\leq (L,\sigma)\) et \(a \in L^n\) tel que \(P_j(a) = 0\), pour tout \(j\leq n\). C'est un équivalent naturel de la dimension de Krull dans le cadre des variétés aux différences.

\textcite[théorème~14.2]{Hru-Frob} déduit du théorème~\ref{Twist LW} la borne suivante sur la dimension totale de \(X\).

\begin{theo}\label{Jacobi}
Soit \(h_{i,j}\) l'ordre de \(x_i\) dans \(P_j\). On a
\[\tdim(X) \leq \max_{\theta\in\fS_n} \sum_{i\leq n}h_{i,\theta(i)}.\]
\end{theo}

C'est l'équivalent pour l'algèbre aux différences d'une conjecture de Jacobi (non résolue) en algèbre différentielle.\medskip

La preuve consiste à prouver, d'abord, un résultat similaire pour les équations polynomiales:
\begin{prop}
Soient \(P_j(x_1,\ldots,x_n)\), où \(1 \leq  j\leq n\), des polynômes sur un corps algébriquement clos \(K\). Soit \(X\) le lieu des zéros des \(P_j\), et \(X_0\) l'union de ses composantes irréductibles de dimension \(0\). Alors
\begin{equation}\label{Bezout}
|X_0(K)| \leq \sum_{\theta\in\fS_n} \prod_{i\leq n} d_{i,\theta(i)},
\end{equation}
où \(d_{i,j}\) est le degré de \(P_j\) en \(x_i\).
\end{prop}

Cette version du théorème de Bézout se démontre aisément en calculant un produit d'intersection.

On relie ensuite cette borne à la borne recherchée par un argument de spécialisation --- cette fois-ci en algèbre aux différences. Soit \((D,\sigma) \leq (K,\sigma)\) un anneau aux différences finiment engendré, bien choisi, qui contient les coefficients des \(P_j\). Soient \(f \colon (D,\sigma) \to (\alg{\Ff}_q,\phi_q)\) un morphisme d'anneaux aux différences et \(Q_j = f(P_j)\). Ce polynôme aux différences s'identifie, puisque \(\phi_q(x) = x^q\), à un polynôme de degré \(d_{i,j}(q)\) en \(x_i\), avec \(d_{i,j}(q)\) de l'ordre de \(m q^{h_{i,j}} + O(q^{h_{i,j} - 1})\), où \(m\in \Zz_{>0}\), quand \(q\) est grand. On a donc
\[\lim_{q\to\infty} \log_q d_{i,j}(q) = h_{i,j}.\]

Par ailleurs, si \(X_q\) dénote le lieu des zéros des \(Q_j\), on peut montrer que \(X_q\) est de dimension zéro et on peut déduire du théorème~\ref{Twist LW}, que pour \(q\) suffisamment grand,
\[|X_q(\alg{\Ff}_q)| = c q^{\tdim(X)} + O(q^{\tdim(X)-1/2}),\] où \(c\in \Qq_{>0}\). Il s'ensuit que
\begin{align*}
\tdim(X) &= \lim_{q\to\infty} \log_q |X_q(\alg{\Ff}_q)|\\
&\leq \lim_{q\to\infty} \log_q \sum_{\theta\in\fS_n} \prod_{i\leq n} d_{i,\theta(i)}(q) & \text{par~(\ref{Bezout}})\\
&= \max_{\theta\in\fS_n} \sum_{i\leq n}h_{i,\theta(i)}.
\end{align*}

%% printbibliography is the command from the package biblatex
\printbibliography

\end{document}

%%% Local Variables:
%%% mode: latex
%%% TeX-master: t
%%% End: